\documentclass[a4paper,10pt]{article}
\usepackage{amsfonts}
\usepackage{bbm}
\usepackage{mathrsfs}
\usepackage[shortlabels]{enumitem}
\usepackage{amsmath,amsthm,amssymb,amscd}
\usepackage{appendix}
\usepackage[center]{titlesec}
\usepackage{xcolor}
\definecolor{revisionblue}{RGB}{0,0,0}
\newcommand{\rev}[1]{\textcolor{revisionblue}{#1}}
\newenvironment{revision}{\begingroup\color{revisionblue}}{\endgroup}
\newtheorem{theo}{Theorem}[section]
\newtheorem{prop}[theo]{Proposition}

\newtheorem{col}[theo]{Corollary}
\theoremstyle{definition}
\newtheorem{defi}[theo]{Definition}

\newtheorem{rem}[theo]{Remark}

\def\theequation{\arabic{section}.\arabic{equation}}

\begin{document}
 \setlength{\baselineskip}{13pt} \pagestyle{myheadings}

\newcommand{\Span}{\operatorname{Span}}

 \title{
 {On the $\mathcal{D}^+_J$ operator on higher-dimensional almost K\"{a}hler manifolds}
 \thanks{ Supported by NSFC (China) Grants 12171417, 1197112.}
 }

 \author{{\large Qiang Tan, Hongyu Wang, Ken Wang\thanks {E-mail:
 kw.dred@gmail.com}, Zuyi Zhang}\\
 }

 \date{}
 \maketitle
 {\centering\it In Memory of Professor Zhengzhong Huang (1916-2012).\par}
~
~

 \noindent {\bf Abstract:}
 \rev{We prove that every smooth exact $J$-invariant two-form on a closed almost K\"{a}hler manifold admits a unique scalar potential, and that the associated generalized Monge-Amp\`ere equation has at most one solution. In real dimension four, \textcolor{black}{we establish} local existence and uniqueness for nearby volume forms with the same total volume under the condition $h_J^-=b^+-1$. We also construct a nonintegrable almost K\"{a}hler four-torus and volume forms on it that cannot be realized by cohomologous $J$-compatible forms. These volume forms have the same total volume as the original volume form and can be chosen arbitrarily close to it. Finally, we relate the resulting trace potential to existing \textcolor{black}{a priori} estimates.}

  \noindent {\bf AMS Classification (2020):} 53D35; 53C56; 53C65; 32Q60.

  \noindent {\bf Keywords: almost K\"{a}hler form, $\mathcal{D}^+_J$ operator,\\
           generalized Monge-Amp\`{e}re equation.}

 \section{Introduction}
   Yau's Theorem for closed K\"{a}hler manifolds \cite{Yau} asserts that one can prescribe the volume form
  of a K\"{a}hler metric within a given K\"{a}hler class.
  This result is fundamental in the theory of K\"{a}hler manifolds (cf. Calabi \cite{Calabi,Calabi2}) and has broad applications in both geometry and mathematical physics (cf. Yau \cite{Yau1}).
\begin{revision}
All manifolds in the statements below are connected, and the almost K\"ahler manifolds have real dimension $2n\ge4$.
\end{revision}
  More precisely, Yau's Theorem is stated as follows:
  Let $(M, \omega, J, g)$ be a closed K\"{a}hler manifold of real dimension $2n$, where $\omega$ denotes the K\"{a}hler form, $J$ is an $\omega$-compatible complex structure, and $g(\cdot, \cdot) = \omega(\cdot, J \cdot)$.
  Given a smooth volume form $e^F \omega^n$ on $M$, there exists a unique smooth function $\varphi$ that satisfies
  \begin{equation}
  \begin{aligned} \label{Monge-Ampere}
       (\omega+\sqrt{-1}\partial_J\bar{\partial}_J\varphi)^n = e^F\omega^n,\\
      \omega+\sqrt{-1}\partial_J\bar{\partial}_J\varphi  >  0, \quad \sup_{M} \varphi=0,
  \end{aligned}
  \end{equation}
  \rev{provided that $F$ satisfies the necessary normalization}
  \begin{equation*}
    \int_{M}e^F\omega^n= \int_{M}\omega^n.
  \end{equation*}
   Equation (\ref{Monge-Ampere}) is known as the complex Monge-Amp\`{e}re equation for K\"{a}hler manifolds.
   \rev{The continuity method requires detailed a priori estimates for the derivatives of $\varphi$ up to third order (cf. Yau \cite{Yau}).}
   \rev{Uniqueness of the solution $\varphi$ follows from an elementary argument of Calabi \cite{Calabi2}.}

  \rev{Another extension of Yau's Theorem concerns a nonintegrable almost complex structure $J$.}
  Recall that an almost complex structure $J$ is said to be tamed by a symplectic form $\omega$ if the bilinear form $\omega(\cdot, J \cdot)$ is positive definite. The structure $J$ is termed compatible (or calibrated) with $\omega$ when the same bilinear form is symmetric,
    that is, when $\omega(\cdot, J \cdot) > 0$ and $\omega(J \cdot, J \cdot) = \omega(\cdot, \cdot)$ (cf. McDuff-Salamon \cite{MS}).
    In the 1990s, Gromov posed the following problem to P. Delano\"{e} \cite{Dela}:
    Let $(M, \omega)$ be a closed symplectic manifold, $J$ an almost complex structure compatible with $\omega$, and $F$ a smooth function on $M$ such that
\begin{equation*}
      \int_{M}e^F\omega^n= \int_{M}\omega^n.
\end{equation*}
  Is it possible to find a smooth function $\varphi$ on $M$ such that $\omega + dJ d\varphi$ is a symplectic form taming $J$ and satisfies
  \begin{equation}\label{Gro CY}
    (\omega+dJd\varphi)^n=e^F\omega^n ?
  \end{equation}
 However, P. Delano\"{e} \cite{Dela} showed that when $n = 2$, the answer to this question is negative. \rev{Wang and Zhu \cite{WZ} later extended this result to all dimensions.}
 They construct a smooth function $\varphi_0$ such that $\omega + dJ d\varphi_0$ lies on the boundary of the set of symplectic forms taming $J$.
  Its $(1,1)$ part is therefore semipositive definite but not strictly positive, yet $(\omega + dJ d\varphi_0)^n > 0$.
  \rev{The strictly positive contribution of the $(2,0) + (0,2)$ part of $dJ d\varphi_0$ accounts for this possibility.}
  \rev{This suggests that the difficulty in Gromov's proposal lies in the 2-form $dJ d\varphi$: it is generally not of type $(1,1)$ with respect to $J$ because $J$ need not be integrable.}
  Its $(1,1)$ part is given by
  \begin{equation*}
  \sqrt{-1}\partial_J\bar{\partial}_J\varphi =\frac{1}{2}(dJd\varphi)^{(1,1)}
\end{equation*}
  which agrees with the standard notation when $J$ is integrable.
  Thus, one can study the Monge-Amp\`{e}re equation for nonintegrable almost complex structures.
  \rev{For results on almost complex manifolds, see Harvey-Lawson \cite{HL}, Tosatti-Wang-Weinkove-Yang \cite{TWWY}, Chu-Tosatti-Weinkove \cite{CTW}, and the references therein.}

     To study the Donaldson tameness problem \cite{DonT},
     Tan, Wang, Zhou, and Zhu introduced the operator $\mathcal{D}^+_J$ on tamed closed almost complex four-manifolds $(M^4, J)$ in \cite{TaWaZhZh}.
\rev{They proved compatibility under the hypothesis $h_J^-=b^+-1$. Here $b^+$ is the positive index of the intersection form and $h_J^-$ is the dimension of the subspace of $H^2(M;\mathbb R)$ represented by closed $J$-anti-invariant forms.}
     The operator $\mathcal{D}^+_J$ can be viewed as a generalization of $\partial_J \bar{\partial}_J$.
     Specifically, when $J$ is integrable, we have $\mathcal{D}^+_J = 2\sqrt{-1} \partial_J \bar{\partial}_J$.
     In \cite{WWZ}, Wang, Wang, and Zhu used the operator $\mathcal{D}^+_J$ to study the Nakai-Moishezon criterion for tamed almost Hermitian $4$-manifolds.
     In \cite{WZZ}, Wang, Zhang, Zheng, and Zhu studied the generalized Monge-Amp\`{e}re equation
       \begin{equation*}
(\omega+\mathcal{D}^+_J(f))^2=e^F\omega^2
\end{equation*}
       on almost K\"{a}hler surfaces.
    \rev{In dimension four, $d^-_Jd^*$ is a self-adjoint, strongly elliptic linear operator and plays a key role in defining $\mathcal{D}^+_J$.}
     However, in higher dimensions, $d^-_Jd^*$ is no longer elliptic.
\begin{revision}
\textcolor{black}{To recover ellipticity in higher dimensions, we use an operator} on one-forms. Let $\mathcal E_J$ be the space of smooth exact $J$-invariant two-forms on a closed almost K\"ahler manifold. Each $\psi\in\mathcal E_J$ has a unique coexact primitive. An elliptic equation \textcolor{black}{with} this primitive \textcolor{black}{as its right-hand side} determines a mean-zero scalar $T_J\psi$ and an anti-invariant two-form $\sigma$. In Appendix~\ref{appendices} we prove that $T_J$ is injective. Its inverse on $\mathcal F_J=T_J(\mathcal E_J)\subseteq C^\infty_0(M)$ satisfies
\[
 \mathcal D_J^+(f)=dd^*(-f\omega+\sigma)=dJdf+dd^*\sigma .
\]
The associated one-form $\mathcal W_J(f)=d^*(-f\omega+\sigma)$ runs through precisely the coexact one-forms $A$ such that $d_J^- A=0$ (Theorem~\ref{W,d-problem}). An arbitrary one-form with $d_J^- A=0$ differs from such a primitive by a closed form.
\end{revision}

    \vskip 6pt



\begin{revision}
    \begin{theo}\label{thm:main}
Let $(M^{2n},\omega,J,g)$ be a closed almost K\"ahler manifold. Every smooth exact $J$-invariant two-form is uniquely represented as $\mathcal{D}_J^+(f)$ for some $f\in\mathcal F_J$.

Moreover, for any smooth function $F$ satisfying
\[
\int_M e^F\omega^n=\int_M\omega^n,
\]
the generalized Monge--Amp\`ere equation
\begin{equation}\label{eq:main-MA}
\begin{cases}
(\omega+\mathcal{D}_J^+(f))^n=e^F\omega^n\\
\ \ \ \omega+\mathcal{D}_J^+(f)>0
\end{cases},
\end{equation}
has at most one solution $f\in\mathcal F_J$.

Furthermore, if $n=2$ and
\begin{equation}\label{equ:1.4}
    h_J^-=b^+(M)-1,
\end{equation}
then there exists $\delta>0$ such that
equation~\eqref{eq:main-MA} admits a unique smooth
solution $f\in\mathcal F_J$ whenever
\[
\|e^F-1\|_{H^4(g)}<\delta,
\]
provided the above volume normalization holds.
    \end{theo}

\textcolor{black}{The uniqueness argument compares compatible forms in the same cohomology class with the same volume. The arithmetic--geometric mean inequality, together with equality of their mixed and total volumes after integration, forces the two forms to agree. Injectivity of the canonical operator then identifies their scalar potentials.}

\textcolor{black}{In dimension four, the traces of harmonic self-dual forms describe the linear obstructions to prescribing volume. When only the constant trace remains, the inverse-function theorem gives local existence for nearby volume forms with the same total volume. A nonconstant harmonic trace supplies the obstruction used in the four-torus example.}

The relation \textcolor{black}{\eqref{equ:1.4}} is necessary. A counterexample is constructed for the 4-torus equipped with a nonintegrable almost complex structure in Theorem \ref{theo:nonexistence} when $h_J^-<b^+(M)-1$.

\end{revision}

   \vskip 6pt

\begin{revision}
For an exact invariant difference $\psi=\omega_1-\omega$ of compatible forms, let $\omega_t=(1-t)\omega+t\omega_1$ and let $\Delta_{g_t}=d_t^*d$. In Section~\ref{elliptical system} we choose the mean-zero solution of $-\Delta_{g_t} f_t=\Lambda_{\omega_t}\psi$. Subtracting $Jdf_t$ from the coexact primitive of $\psi$ gives a one-form $a_t$ with $\psi=dJdf_t+da_t$ and $da_t$ primitive for $\omega_t$. At $t=1$ the potential $2(f_1-\sup_M f_1)$ gives the normalization in the estimate of Tosatti, Weinkove and Yau \cite{TWY}. The potential operator is also used to study energy functionals on almost K\"ahler manifolds \cite{WZZa,WZZb}.
\end{revision}

   \vskip 6pt

   \rev{Finally, as in the K\"{a}hler case \cite{DonK}, we ask whether a fixed symplectic class on a compact almost K\"{a}hler manifold admits each of four kinds of special metrics:} \textcolor{black}{extremal} metrics\textcolor{black}{,} \textcolor{black}{constant} Hermitian scalar curvature metrics\textcolor{black}{,} Hermite-Einstein metrics with $c_1=0$, $c_1<0$ \textcolor{black}{or} $c_1>0$\textcolor{black}{, and} \rev{generalized almost K\"ahler--Ricci solitons.}

 \section{Preliminaries}\setcounter{equation}{0}\label{preliminaries}

    \rev{Let $M$ be a manifold with almost complex structure $J$. For any $x \in M$, $T_xM\otimes_\mathbb{R}\mathbb{C}$, the complexification of $T_xM$, decomposes as}
   \begin{equation}\label{2eq1}
  T_xM\otimes_\mathbb{R}\mathbb{C}=T^{1,0}_x+T^{0,1}_x,
  \end{equation}
  where $T^{1,0}_x$ and $T^{0,1}_x$ are the eigenspaces of $J$ corresponding to the eigenvalues $\sqrt{-1}$ and $-\sqrt{-1}$, respectively.
   A complex tangent vector is of type $(1,0)$ (resp. $(0,1)$) if it belongs to $T^{1,0}_x$ (resp. $T^{0,1}_x$).
   Let $TM\otimes_\mathbb{R}\mathbb{C}$ be the complexification of the tangent bundle.
   Similarly, let $T^*M\otimes_{\mathbb{R}}\mathbb{C}$ denote the complexification of the cotangent bundle $T^*M$.
   \rev{The action of $J$ on $T^*M\otimes_{\mathbb{R}}\mathbb{C}$ is defined by}
   \begin{equation*}
\forall \alpha\in T^*M\otimes_{\mathbb{R}}\mathbb{C},\quad J\alpha(\cdot)=-\alpha(J\cdot).
\end{equation*}
   \rev{The eigenspace decomposition of $T^*M\otimes_{\mathbb{R}}\mathbb{C}$ for eigenvalues $\mp\sqrt{-1}$ is}
   \begin{equation}\label{2eq2}
   T^*M\otimes_{\mathbb{R}}\mathbb{C}=\Lambda^{1,0}_J\oplus\Lambda^{0,1}_J.
   \end{equation}
   \rev{Define the exterior bundle $\Lambda^{p,q}_J=\Lambda^p\Lambda^{1,0}_J\otimes\Lambda^q\Lambda^{0,1}_J$.}
   Let $\Omega^{p,q}_J(M)$ denote the space of $C^\infty$ sections of the bundle $\Lambda^{p,q}_J$.
   \rev{For sections in $\Omega^{p,q}_J$, the exterior derivative satisfies}
   \begin{equation}\label{2eq3}
   d\Omega^{p,q}_J\subset\Omega^{p-1,q+2}_J+\Omega^{p+1,q}_J+\Omega^{p,q+1}_J+\Omega^{p+2,q-1}_J.
   \end{equation}
  \rev{Splitting the exterior derivative $d$ by bidegree gives}
  \begin{equation}\label{2eq4}
  d=A_J\oplus\partial_J\oplus\bar{\partial}_J\oplus\bar{A}_J,
  \end{equation}
  \rev{where the components are derivations of bidegrees}
  \begin{equation*}
  |A_J|=(-1,2),\,\,|\partial_J|=(1,0),\,\,|\bar{\partial}_J|=(0,1),\,\,|\bar{A}_J|=(2,-1).
\end{equation*}
  The operators $\partial_J$ and $\bar{\partial}_J$ are of first order, while $A_J$ and $\bar{A}_J$ are of zero order (see \cite{NW}).

   \vskip 6pt

\rev{Throughout, $(M,\omega,J,g)$ is closed and connected, of real dimension $2n$ with $n\ge2$, and $g(X,Y)=\omega(X,JY)$. We write $C^\infty_0(M)$ for the smooth real functions of zero $g$-mean.}
     Let $\Omega^2_\mathbb{R}(M)$ denote the space of real smooth $2$-forms on $M$, i.e.,
  the real $C^\infty$ sections of the bundle $\Lambda^2_\mathbb{R}(M)$.
    The almost complex structure $J$ acts on the space $\Omega_\mathbb{R}^2$ as an involution, defined by
  \begin{equation}\label{involution}
    \alpha \longmapsto \alpha(J\cdot,J\cdot), \quad \alpha\in\Omega_\mathbb{R}^2(M).
  \end{equation}
  This induces a decomposition of 2-forms into $J$-invariant and $J$-anti-invariant parts (see \cite{DonT}):
  \begin{equation*}
  \Omega_\mathbb{R}^2 = \Omega^+_J \oplus \Omega^-_J, \quad \alpha = \alpha_J^+ + \alpha_J^-
  \end{equation*}
  and the corresponding decomposition of vector bundles:
  \begin{equation}\label{J decomposition}
  {\Lambda}_\mathbb{R}^2={\Lambda}_J^+ \oplus {\Lambda}_J^-.
  \end{equation}
   \rev{Define the operators}
   \begin{align}\label{2eq23}
  d^+_J&=\pi_J^+ d: \quad\Omega_\mathbb{R}^1\longrightarrow\Omega^+_J, \nonumber\\
  d^-_J&=\pi_J^- d: \quad\Omega_\mathbb{R}^1\longrightarrow\Omega^-_J,
  \end{align}
  where $\pi_J^\pm: \Omega_\mathbb{R}^2\longrightarrow\Omega^\pm_J$ are the projections onto the $J$-invariant and $J$-anti-invariant parts, respectively.
  A differential $k$-form $B_k$ with $k\leq n$ is called primitive if $L^{n-k+1}_{\omega}B_k=0$, or equivalently, $\Lambda_{\omega} B_k=0$ (see \cite{TY,YanH}).
  Here, $L_{\omega}$ is the Lefschetz operator (see \cite{BrA,TY,YanH}), which acts on a $k$-form $A_k\in\Omega_\mathbb{R}^k(M)$ by
  \begin{equation*}
L_{\omega}(A_k)=\omega\wedge A_k.
\end{equation*}
  \rev{The dual Lefschetz operator $\Lambda_{\omega}:\Omega_\mathbb{R}^k(M)\rightarrow \Omega_\mathbb{R}^{k-2}(M)$ is the contraction map associated with the symplectic form $\omega$.}
    \rev{Denote the space of primitive $k$-forms by $\Omega^k_{0}(M)$.}
    \rev{In degree two,
}    \begin{equation*}
\Omega^2_{0}(M)=\{\alpha\in \Omega_\mathbb{R}^2(M) \,|\, \omega^{n-1}\wedge\alpha=0\}.
\end{equation*}
    Thus,
     \begin{equation*}
\Omega^2_{0}(M)=\Omega^-_J(M)\oplus \Omega^+_{J,0}(M)
\end{equation*}
     and
     \begin{equation}\label{db decomposition}
       \Omega_\mathbb{R}^2(M) = \Omega^2_1(M)\oplus \Omega^-_J(M)\oplus \Omega^+_{J,0}(M),
     \end{equation}
     where $\Omega^+_{J,0}(M)$ is the space of primitive $J$-invariant $2$-forms and
     \begin{equation*}
\Omega^2_1(M):=\{f\omega \,|\, f\in C^\infty(M,\mathbb{R}) \}.
\end{equation*}

   \vskip 6pt

         Define
$$
\mathcal{E}_J=\{\psi\in\Omega_J^+(M):\psi\ \text{is exact}\}.
$$
\begin{revision}
For $\psi\in\mathcal E_J$, let $a=d^*G\psi$ be its coexact primitive, where $G$ is the Hodge Green operator on two-forms. Appendix~\ref{appendices} constructs a function $T_J\psi=f_\psi$ that depends linearly on this primitive. More precisely, let $Q$ be the orthogonal projection onto $\mathbb R\omega\oplus\Lambda_J^-$ and let $\theta_\psi$ be the solution orthogonal to harmonic one-forms of
\[
 (d^*Qd+dd^*)\theta_\psi=a.
\]
Then $f_\psi=-\Lambda_\omega d\theta_\psi/n$ has mean zero. With $\sigma_\psi=\pi_J^-d\theta_\psi$ one has
\[
 d^*(-f_\psi\omega+\sigma_\psi)=a,\qquad da=\psi .
\]
The injectivity proved in Appendix~\ref{appendices} \textcolor{black}{implies that} $f_\psi$ \textcolor{black}{uniquely determines the corresponding element of} $\mathcal E_J$. \textcolor{black}{We therefore define the potential operator by inverting this assignment on its image.}
\end{revision}
\begin{revision}
\begin{defi}\label{D operator defi}
Set $\mathcal F_J=T_J(\mathcal E_J)\subseteq C^\infty_0(M)$. For $f=T_J\psi$ define the canonical potential operator and its one-form primitive by
\[
 \mathcal D_J^+(f)=\psi,\qquad \mathcal W_J(f)=d^*G\psi.
\]
Thus $\mathcal D_J^+=T_J^{-1}$ on $\mathcal F_J$.
\end{defi}
Since $*\omega=\omega^{n-1}/(n-1)!$ and $d\omega=0$, we have $d^*\omega=0$ and
$d^*(f\omega)=-\iota_{\nabla f}\omega=-Jdf$.
In particular, with $\sigma=\sigma_\psi$,
\[
 \mathcal W_J(f)=Jdf+d^*\sigma,\qquad
 \mathcal D_J^+(f)=dJdf+dd^*\sigma .
\]
The anti-invariant two-form satisfies $d_J^-d^*\sigma=-d_J^-Jdf$. If two such forms differ by $\tau\in\Omega_J^-$, then
$\|d^*\tau\|^2=\langle\tau,d_J^-d^*\tau\rangle=0$.
\textcolor{black}{Thus the choice of anti-invariant two-form does not affect the} primitive\textcolor{black}{, which} depends only on $f$.
\end{revision}
\begin{revision}
\begin{rem}
When $J$ is integrable, $dJdf=2\sqrt{-1}\partial_J\bar\partial_Jf$ is invariant. The \textcolor{black}{anti-invariant} equation is \textcolor{black}{therefore} homogeneous, so $d^*\sigma=0$. For any $u\in C^\infty_0(M)$, apply the representation to $\psi=dJdu$. Its canonical potential $f$ satisfies $dJd(f-u)=0$. Tracing gives the scalar Laplace equation for $f-u$, whose mean is zero. Hence $f=u$, and $\mathcal F_J=C^\infty_0(M)$. For general $J$, the full set of smooth mean-zero scalar inputs for which the anti-invariant equation is solvable is
\begin{equation}\label{eq:admissible-domain}
 \widehat{\mathcal F}_J=
 \left\{u\in C^\infty_0(M):
 -d_J^-Jdu\in d_J^-d^*(\Omega_J^-(M))\right\}.
\end{equation}
For such a $u$, choose $\sigma\in\Omega_J^-(M)$ with
$d_J^-d^*\sigma=-d_J^-Jdu$. The exact invariant form
$dJdu+dd^*\sigma$ is independent of this choice, since any two choices have the same codifferential by the preceding adjunction argument. The canonical construction gives
$\mathcal F_J\subseteq\widehat{\mathcal F}_J$.
When $J$ is integrable, both spaces equal $C^\infty_0(M)$.
In real dimension four, the Fredholm argument in Section~\ref{local theo} gives
$\widehat{\mathcal F}_J=C^\infty_0(M)$ even when $J$ is nonintegrable.
The canonical space selects a unique scalar representative for each exact invariant form; the full admissible domain may contain distinct scalars representing the same form.
\end{rem}
\end{revision}

    \vskip 6pt

\begin{revision}
\begin{theo}\label{W,d-problem}
For smooth real one-forms the range of $\mathcal W_J$ is
\[
 \mathcal W_J(\mathcal F_J)=\ker d_J^-\cap d^*\Omega^2_{\mathbb R}(M).
\]
\end{theo}
\begin{proof}
An output $a=\mathcal W_J(f)$ is coexact and $da=\mathcal D_J^+(f)$ is invariant. Conversely, suppose $a$ is coexact and $d_J^-a=0$. Then $da\in\mathcal E_J$. Its canonical coexact primitive is $d^*Gda=a$: the difference of the two primitives is both closed and coexact. Hence $\mathcal W_J(T_J(da))=a$.
\end{proof}
For arbitrary $A\in\ker d_J^-$ the same construction gives $A$ modulo a closed one-form. This qualification is necessary even on a flat complex torus: a nonzero exact one-form lies in $\ker d_J^-$ and is orthogonal to every coexact form.
\end{revision}
\begin{revision}
\begin{col}\label{exact lemma}
Every smooth exact $J$-invariant two-form $\psi$ is $\mathcal D_J^+(f)$ for a unique $f\in\mathcal F_J$. Its coexact primitive is $\mathcal W_J(f)$.
\end{col}
\begin{proof}
Take $f=T_J\psi$. The defining identities and uniqueness follow from Proposition~\ref{prop:elliptic-representation} and the injectivity of $T_J$.
\end{proof}
\end{revision}
\begin{revision}
For the $L^2$ realization let $\mathcal E_J^0$ consist of the $L^2$ invariant two-forms that equal $db$ for some real one-form $b\in H^1$. For $\psi\in\mathcal E_J^0$ the construction in Appendix~\ref{appendices} defines $T_J\psi$ by the same formula as in the smooth case. Set $\mathcal F_J^2=T_J(\mathcal E_J^0)$.

\begin{prop}\label{closed range 2}
The space $\mathcal F_J^2$ lies in $H^2(M)$. With the norm $\|T_J\psi\|_{\mathcal F_J^2}=\|\psi\|_{L^2}$ it is a Hilbert space, and $\mathcal D_J^+\colon\mathcal F_J^2\to L^2(\Lambda_J^+)$ has closed range $\mathcal E_J^0$.
\end{prop}
\begin{proof}
The bounded Hodge projection $dd^*G$ on $L^2$ has range $dH^1(\Lambda^1)$: it fixes every exact form and factors through $d^*G\colon L^2\to H^1$. Intersecting this range with the closed subspace of invariant forms gives the closed subspace $\mathcal E_J^0$.

Given $\psi\in\mathcal E_J^0$, its primitive $a=d^*G\psi$ lies in $H^1$ and is orthogonal to harmonic one-forms. Elliptic regularity for $d^*Qd+dd^*$ gives $\theta_\psi\in H^3$ and hence $f_\psi,\sigma_\psi\in H^2$. Applying $d^*$ to the equation for $\theta_\psi$ gives $d^*\theta_\psi=0$, as in the smooth case. Thus $d^*(-f_\psi\omega+\sigma_\psi)=a$ weakly. If $f_\psi=0$, adjunction gives $\|a\|_{L^2}^2=\langle\sigma_\psi,\psi\rangle_{L^2}=0$. Therefore $T_J$ is injective on $\mathcal E_J^0$. The asserted norm identifies its image isometrically with $\mathcal E_J^0$\textcolor{black}{. Since the latter is a closed subspace of a Hilbert space, the image is complete}, and $T_J^{-1}$ extends the smooth operator.
\end{proof}
\end{revision}
 \section{Generalized Monge-Amp\`{e}re Equation on Almost K\"{a}hler Manifolds}\setcounter{equation}{0}\label{local theo}

\rev{   Monge-Amp\`{e}re equations form an important class of fully nonlinear equations with deep connections to analysis and geometry (see \cite{FWW,Yau1}). Extensive work on their existence, uniqueness and regularity under different hypotheses includes Tosatti-Weinkove \cite{ToWe}, Chu-Tosatti-Weinkove \cite{CTW}, Demailly-Pali \cite{DP}, Pli\'{s} \cite{Pli}, Zhang-Zhang \cite{ZZ} and others.}
\rev{On a closed almost K\"ahler manifold $(M,\omega,J,g)$, consider}
     \begin{equation}\label{generalized MA}
       (\omega+\mathcal{D}^+_J(f))^n=e^F\omega^n
     \end{equation}
\rev{for $f\in\mathcal F_J$ with}
       \begin{equation*}
       \omega+\mathcal{D}^+_J(f)>0,
\end{equation*}
      where $F\in C^\infty(M)$ satisfies
      \begin{equation*}
      \int_{M}\omega^n=\int_{M}e^F\omega^n.
     \end{equation*}
     If $J$ is integrable, $\mathcal{D}^+_J(f)=2\sqrt{-1}\partial_J\bar{\partial}_Jf$.
\rev{In this case \eqref{generalized MA} is the complex Monge--Amp\`ere equation with K\"ahler potential $2f$ \cite{Calabi,Yau}.}

     \rev{The positivity condition determines the potential space}
\begin{equation}\label{eq:positive-potential-space}
\mathcal{F}_J^+
=
\{f\in\mathcal{F}_J:\omega+\mathcal{D}^+_J(f)>0\}.
\end{equation}
For $f\in\mathcal{F}_J^+$, set
\begin{equation}\label{eq:omega-f}
\omega_f=\omega+\mathcal{D}^+_J(f).
\end{equation}
Since $\mathcal{D}^+_J(f)$ is exact and $J$-invariant, $\omega_f$ is a closed $J$-compatible symplectic form satisfying $[\omega_f]=[\omega]$.

Define the nonlinear volume map
\begin{equation}\label{eq:volume-map}
M_J(f)
=
\textcolor{black}{\frac{\omega_f^n}{\omega^n}}.
\end{equation}
\rev{In these coordinates, the generalized Monge--Amp\`ere equation takes the form}
\begin{equation}\label{eq:generalized-MA-canonical}
\textcolor{black}{F=\log M_J(f)},\qquad f\in\mathcal{F}_J^+.
\end{equation}
This is a well-defined equation on the canonical potential space. Notice that
$$
\int_M \textcolor{black}{M_J(f)}\omega^n
=
\int_M \omega_f^n
=
\int_M \omega^n.
$$
\textcolor{black}{Thus every $F=\log M_J(f)$ satisfies the normalization condition.}

The key question becomes:
\begin{equation}\label{eq:surjectivity-question}
\text{Is }
M_J\colon
\mathcal{F}_J^+
\longrightarrow
\left\{
\textcolor{black}{v\in C^\infty(M):v>0,\ \int_M v\omega^n=\int_M\omega^n}
\right\}
\text{ surjective?}
\end{equation}
\rev{We begin with the linearization of $M_J$.}

 \rev{Fix $f\in\mathcal{F}_J^+$ and a direction $u\in\mathcal{F}_J$. Linearity of $\mathcal{D}^+_J$ gives}
$$
\omega_{f+tu}
=
\omega_f+t\mathcal{D}^+_J(u).
$$
Differentiating the volume form gives
    $$
\frac{d}{dt}\bigg|_{t=0}\omega_{f+tu}^n
=
n\omega_f^{n-1}\wedge \mathcal{D}^+_J(u).
$$
Therefore
$$
DM_J(f)[u]
=
\frac{d}{dt}\bigg|_{t=0}
\textcolor{black}{\frac{\omega_{f+tu}^n}{\omega^n}}
=
\frac{n\omega_f^{n-1}\wedge \mathcal{D}^+_J(u)}{\textcolor{black}{\omega^n}}.
$$
Since $\mathcal{D}^+_J(u)$ is $J$-invariant, \textcolor{black}{taking its trace} with respect to $\omega_f$ \textcolor{black}{gives}
\begin{equation}\label{eq:linearization}
DM_J(f)[u]
=
\textcolor{black}{M_J(f)}\Lambda_{\omega_f}\mathcal{D}^+_J(u).
\end{equation}
At $f=0$, we have \textcolor{black}{$M_J(0)=1$ and} $\omega_f=\omega$, so
\begin{equation}\label{eq:linearization-zero}
DM_J(0)[u]
=
\Lambda_\omega \mathcal{D}^+_J(u).
\end{equation}
\rev{The bijection $\mathcal{D}^+_J\colon\mathcal{F}_J\to\mathcal{E}_J$ identifies the range of the linearization as}
\begin{equation}\label{eq:RJ}
\mathcal{R}_J
=
\{\Lambda_\omega\psi:\psi\in\mathcal{E}_J\}.
\end{equation}
\rev{Surjectivity of the linearization amounts to}
\begin{equation}\label{eq:RJ-full}
\mathcal{R}_J
=
\left\{
h\in C^\infty(M):\int_M h\,dV_g=0
\right\}.
\end{equation}
\rev{The four-dimensional calculation below gives a criterion for this surjectivity and an example in which it fails.}

    \vskip 6pt

\rev{  \textcolor{black}{Before addressing local existence, we} prove uniqueness on a closed almost K\"{a}hler manifold $(M,\omega,J,g)$.}
   \begin{prop}\label{prop:uniqueness-exact-J-invariant}
Let $(M,\omega,J,g)$ be closed almost K\"ahler. Suppose $\psi_0,\psi_1\in\mathcal{E}_J$ satisfy $\omega+\psi_i>0$ and
$$
(\omega+\psi_0)^n=(\omega+\psi_1)^n.
$$
Then
\begin{equation}\label{eq:uniqueness-exact-J-invariant}
\psi_0=\psi_1.
\end{equation}
\end{prop}

\begin{proof}
Set $\Omega_i=\omega+\psi_i$. Since both $\psi_i$ are exact,
$$
[\Omega_0]=[\Omega_1]=[\omega].
$$
At a fixed point, diagonalize $\Omega_1$ relative to $\Omega_0$. Let the positive eigenvalues be $\lambda_1,\dots,\lambda_n$.
The equality $\Omega_1^n=\Omega_0^n$ implies
\begin{equation}\label{eq:product-lambda}
\prod_{j=1}^n\lambda_j=1.
\end{equation}
By the arithmetic--geometric mean inequality,
$$
\frac{1}{n}\sum_{j=1}^n\lambda_j
\ge
\left(\prod_{j=1}^n\lambda_j\right)^{1/n}
=1.
$$
\rev{The corresponding mixed-volume ratio is}
$$
\frac{\Omega_1\wedge\Omega_0^{n-1}}{\Omega_0^n}
=
\frac{1}{n}\sum_{j=1}^n\lambda_j.
$$
Consequently,
\begin{equation}\label{eq:volume-inequality}
\Omega_1\wedge\Omega_0^{n-1}
\ge
\Omega_0^n.
\end{equation}
Since $\Omega_1$ and $\Omega_0$ are cohomologous and closed,
$$
\int_M \Omega_1\wedge\Omega_0^{n-1}
=
\int_M \Omega_0^n.
$$
\rev{The nonnegative difference in \eqref{eq:volume-inequality} therefore has integral zero and vanishes pointwise:}
$$
\frac{1}{n}\sum_{j=1}^n\lambda_j=1.
$$
Equality in the arithmetic--geometric mean inequality holds only when all the $\lambda_j$ are equal. By \eqref{eq:product-lambda}, they must all equal $1$. Therefore $\Omega_1=\Omega_0$, so $\psi_1=\psi_0$.
This proves the proposition.
   \end{proof}

   Suppose $f_0,f_1\in\mathcal{F}_J^+$ both satisfy
$$
(\omega+\mathcal{D}^+_J(f_i))^n
=
e^F\omega^n.
$$
By Proposition \ref{prop:uniqueness-exact-J-invariant},
$$
\mathcal{D}^+_J(f_0)
=
\mathcal{D}^+_J(f_1).
$$
Since $\mathcal{D}^+_J$ is injective on $\mathcal{F}_J$,
\begin{equation}\label{eq:f0-equals-f1}
f_0=f_1.
\end{equation}
\rev{If we extend the canonical operator to $\mathcal F_J\oplus\mathbb R$ by letting it vanish on constants, the solutions are unique modulo constants.}

\begin{revision}
\begin{theo}\label{uniqueness}
For a fixed $F$, equation \eqref{generalized MA} has at most one solution in $\mathcal F_J^+$. Its compatible form is the unique form in $[\omega]$ compatible with $J$ and having volume $e^F\omega^n$, whenever such a form exists.
\end{theo}
\end{revision}

\rev{We next study local existence in dimension four.}
Let $(M,\omega,J,g)$ be a closed almost K\"{a}hler four-manifold.
The four-dimensional Hodge star gives the splitting $$\Lambda^2 = \Lambda_g^+ \oplus \Lambda_g^-,\,\,\,\Lambda_g^\pm = \{\alpha : *\alpha = \pm\alpha\}.$$
\rev{The Hodge and almost complex decompositions are related by}
\begin{equation}\label{deco 4-dim}
  \Lambda_g^+ = \mathbb{R}\omega \oplus \Lambda_J^-, \qquad \Lambda_J^+ = \mathbb{R}\omega \oplus \Lambda_g^-.
\end{equation}
Define
$$
\mathcal{E}_J^+=\{\psi\in\mathcal{E}_J:\omega+\psi>0\},\qquad
V(\psi)=\frac{(\omega+\psi)^2}{\omega^2}.
$$
At $\psi=0$, $V(0)=1$. For any $\eta\in\mathcal{E}_J$,
$$
V(t\eta)=1+t\frac{2\omega\wedge\eta}{\omega^2}+t^2\frac{\eta^2}{\omega^2}.
$$
Hence
$$
DV(0)[\eta]=\frac{2\omega\wedge\eta}{\omega^2}=\Lambda_\omega\eta.
$$
\rev{The prescribed-volume equation thus linearizes to}
\begin{equation}\label{eq:linearized-prescribed-volume}
\Lambda_\omega\eta=h,\qquad\eta\in\mathcal{E}_J.
\end{equation}
The function $h$ must have zero mean because for $\eta=da$,
$$
\int_M \Lambda_\omega\eta\,dV_g
=
\int_M \eta\wedge\omega
=
0.
$$
\textcolor{black}{To describe the range of the trace map, let}
$$
\mathcal{H}_g^+=\{\beta\in\Omega^2(M):*\beta=\beta,\ d\beta=0\}
$$
\rev{be the space of harmonic self-dual two-forms, of dimension $b^+$. Write $\mathcal H_J^-$ for the closed $J$-anti-invariant two-forms. They are self-dual and hence harmonic, so $\dim\mathcal H_J^-=h_J^-$.}

\begin{prop}\label{linearized-image}
A smooth function $h$ has the form $h=\Lambda_\omega\psi$ for some exact $J$-invariant two-form $\psi$ if and only if
\begin{equation}\label{eq:linearized-image}
\int_M h\langle\omega,\beta\rangle_g\,dV_g=0 \quad \text{for every } \beta\in\mathcal{H}_g^+.
\end{equation}
\end{prop}
  \begin{proof}
 Suppose $\psi=da$ is $J$-invariant and $\Lambda_\omega\psi=h$. By (\ref{deco 4-dim}), its self-dual part is
\begin{equation}\label{eq:self-dual-part}
\psi_g^+=\frac{h}{2}\omega.
\end{equation}
If $\beta\in\mathcal{H}_g^+$, then $\langle\psi,\beta\rangle=\langle\psi_g^+,\beta\rangle$, hence
$$
\begin{aligned}
\frac{1}{2}\int_M h\langle\omega,\beta\rangle\,dV_g
&=\int_M \langle\psi_g^+,\beta\rangle\,dV_g \\
&=\int_M \langle da,\beta\rangle\,dV_g \\
&=\int_M \langle a,d^*\beta\rangle\,dV_g=0.
\end{aligned}
$$
Thus \eqref{eq:linearized-image} is necessary.

\textcolor{black}{For the converse, suppose} $h$ satisfies \eqref{eq:linearized-image}. We seek a one-form $a$ with
\begin{equation}\label{eq:dga}
d_g^+ a = \frac{h}{2}\omega,
\end{equation}
where $d_g^+ = \pi_g^+ \circ d$. Consider the elliptic complex
\begin{equation}\label{eq:elliptic-complex}
0 \longrightarrow \Omega^0(M) \overset{d}{\longrightarrow} \Omega^1(M) \overset{d_g^+}{\longrightarrow} \Omega_g^+(M) \longrightarrow 0.
\end{equation}
\rev{For a nonzero covector $\xi$, the symbols are $c\mapsto c\xi$ and $a\mapsto\pi_g^+(\xi\wedge a)$. In a positively oriented orthonormal coframe with $\xi=|\xi|e^1$, the latter map has rank three and kernel $\mathbb R e^1$. Thus the symbol sequence is exact.}
By elliptic Hodge theory,
\begin{equation}\label{eq:Hodge-theory-image}
\operatorname{Ran} d_g^+ = (\mathcal{H}_g^+)^\perp.
\end{equation}
Indeed, on a self-dual two-form $\beta$, $d_g^{+*}\beta = d^*\beta$. Self-duality and $d^*\beta = 0$ together imply $d\beta = 0$, so $\ker d_g^{+*} = \mathcal{H}_g^+$.
Our assumption is exactly $(h/2)\omega \perp \mathcal{H}_g^+$. Thus the Fredholm alternative provides a smooth $a$ satisfying \eqref{eq:dga}.
Put $\psi = da$. Then $\psi_g^+ = (h/2)\omega$ and the remaining part $\psi_g^-$ is anti-self-dual; by (\ref{deco 4-dim}), $\psi$ is $J$-invariant.
\rev{Taking the trace gives}
$$
\Lambda_\omega\psi = \frac{h}{2}\Lambda_\omega\omega = h,
$$
since $\Lambda_\omega\omega = 2$. This proves sufficiency and establishes
$$
\Lambda_\omega (\mathcal{E}_J) = \left\{ h \in C^\infty(M) : \int_M h\langle\omega,\beta\rangle_g \, dV_g = 0 \quad \forall \beta \in \mathcal{H}_g^+ \right\}.
$$
 \end{proof}

\textcolor{black}{To count the independent obstruction functions, define}
$$
T_\omega \colon \mathcal{H}_g^+ \longrightarrow C^\infty(M), \qquad T_\omega(\beta) = \langle \omega, \beta \rangle_g.
$$
If \(T_\omega(\beta) = 0\), then \(\beta\) is harmonic, self-dual, and pointwise orthogonal to \(\omega\). By (\ref{deco 4-dim}), this means that \(\beta\) is \(J\)-anti-invariant. Conversely, every harmonic \(J\)-anti-invariant form is self-dual and orthogonal to \(\omega\). Hence
\begin{equation}\label{eq:ker-Tomega}
\ker T_\omega = \mathcal{H}_J^-.
\end{equation}
The rank-nullity theorem yields
\begin{equation}\label{eq:dim-ran-Tomega}
\dim \operatorname{Ran} T_\omega = b^+ - h_J^-.
\end{equation}
Since \(\omega \in \mathcal{H}_g^+\) and \(T_\omega(\omega) = 2\), the image always contains the constants.
Thus the number of independent nonconstant obstruction functions is \(b^+ - h_J^- - 1\).
Proposition \ref{linearized-image} gives
\begin{equation}\label{eq:codim}
\operatorname{codim}_{C_0^\infty(M)} \Lambda_\omega(\mathcal{E}_J) = b^+ - h_J^- - 1.
\end{equation}
Therefore the trace map is onto the mean-zero functions precisely when the codimension is zero:
\begin{equation}\label{eq:onto-condition}
\Lambda_\omega(\mathcal{E}_J) = C_0^\infty(M) \iff h_J^- = b^+ - 1.
\end{equation}

\begin{revision}
This obstruction space also describes the ambiguity of the original scalar potential. Write $\widehat{\mathcal D}_J^+$ for the four-dimensional operator \textcolor{black}{of} \cite{TaWaZhZh,WWZ} on \textcolor{black}{smooth mean-zero functions, extended to} all smooth real functions \textcolor{black}{by setting $\widehat{\mathcal D}_J^+(1)=0$}. In our convention it is $d\,d^*(-f\omega+\sigma)$, where $d_J^-d^*\sigma=-d_J^-Jdf$. The self-adjoint operator $d_J^-d^*$ on $\Omega_J^-$ is elliptic in dimension four. For a self-dual form $\beta$, ellipticity follows from the symbol identity $|\iota_{\xi^\sharp}\beta|^2=|\xi|^2|\beta|^2/2$. By adjunction\textcolor{black}{,} its kernel consists of the coclosed anti-invariant forms\textcolor{black}{. Applying adjunction once more shows that} the right-hand side is orthogonal to \textcolor{black}{this} kernel. Thus the equation admits a smooth solution, and any two solutions have the same codifferential.

If $\widehat{\mathcal D}_J^+ f=0$, the one-form $d^*(-f\omega+\sigma)$ is both closed and coexact, hence zero. The self-dual form $-f\omega+\sigma$ is then coclosed and therefore harmonic. Conversely, decomposing any $\beta\in\mathcal H_g^+$ into its $\omega$ and anti-invariant components gives a kernel element $f=-T_\omega(\beta)/2$. Consequently
\[
 \ker\widehat{\mathcal D}_J^+=-\tfrac12 T_\omega(\mathcal H_g^+).
\]
On mean-zero functions this kernel has dimension $b^+-h_J^--1$. Each mean-zero $u$ has a unique canonical representative $T_J(\widehat{\mathcal D}_J^+u)$ producing the same two-form; their difference lies in this kernel. Hence
\begin{equation}\label{eq:scalar-domain-splitting}
 C^\infty_0(M)=\mathcal F_J\oplus
 \ker\bigl(\widehat{\mathcal D}_J^+|_{C^\infty_0(M)}\bigr).
\end{equation}
In particular, $\mathcal F_J=C^\infty_0(M)$ in dimension four if and only if $h_J^-=b^+-1$. The trace image $\Lambda_\omega(\mathcal E_J)$ in \eqref{eq:RJ} and the canonical image $T_J(\mathcal E_J)$ arise from different maps; the canonical potential is defined by the latter.
\end{revision}
Equivalently, if \(h_J^- = b^+ - 1\), then each \(\beta \in \mathcal{H}_g^+\) decomposes as \(\beta = c\omega + \gamma\) with \(\gamma \in \mathcal{H}_J^-\), so \(\langle \omega, \beta \rangle = 2c\) is constant. Conversely, if the trace of every harmonic self-dual form is constant, subtracting the corresponding multiple of \(\omega\) produces an element of \(\mathcal{H}_J^-\). Hence
$$
\begin{aligned}
h_J^- = b^+ - 1
&\iff \text{all harmonic self-dual forms have constant trace} \\
&\iff \text{the linearized volume equation is surjective.}
\end{aligned}
$$
\begin{revision}
\begin{theo}\label{local-volume}
Let $(M,\omega,J,g)$ be a closed almost K\"ahler four-manifold with $h_J^-=b^+-1$. There is a $\delta>0$ such that every smooth positive function $\rho$ with
\[
 \int_M\rho\,\omega^2=\int_M\omega^2,\qquad
 \|\rho-1\|_{H^4(g)}<\delta
\]
is realized by a unique smooth $J$-compatible form $\Omega$ satisfying
\begin{equation}\label{eq:realized}
 [\Omega]=[\omega],\qquad \Omega^2=\rho\,\omega^2 .
\end{equation}
For $\rho=e^F$, this gives a unique solution of \eqref{generalized MA} in $\mathcal F_J^+$.
\end{theo}
\end{revision}
\begin{proof}
\rev{Set $s=4$. The Sobolev space $H^s(M)$ embeds into $C^0(M)$ and is a Banach algebra. Let $\mathcal E_J^s$ be the closed subspace of $H^s$ exact $J$-invariant two-forms (the exact subspace is the range of the Hodge projection $dd^*G$), and set}
$$
H_0^s(M) = \left\{ h \in H^s(M) : \int_M h \, dV_g = 0 \right\}.
$$
The set
$$
\mathcal{U} = \{\psi \in \mathcal{E}_J^s : \omega + \psi > 0\}
$$
is open and contains zero. Define
\begin{equation}\label{eq:Phi-def}
\Phi(\psi) = \frac{(\omega + \psi)^2}{\omega^2} - 1.
\end{equation}
Since $\psi$ is exact and $d\omega = 0$,
$$
\int_M \omega \wedge \psi = 0, \qquad \int_M \psi \wedge \psi = 0.
$$
Hence $\Phi(\psi)$ has mean zero. The Banach algebra property implies that $\Phi \colon \mathcal{U} \to H_0^s(M)$ is smooth. Its derivative is
$$
D\Phi(0)[\eta] = \Lambda_\omega \eta.
$$
\rev{\textcolor{black}{We first check injectivity.} If $\Lambda_\omega\eta=0$ for $\eta\in\mathcal E_J^s$, then $\eta$ is anti-self-dual. Since it is exact, Stokes' theorem gives $\|\eta\|_{L^2}^2=-\int_M\eta\wedge\eta=0$. For \textcolor{black}{surjectivity, let} $h\in H_0^s(M)$\textcolor{black}{. Our} hypothesis implies that $h\omega$ is orthogonal to $\mathcal H_g^+$. The elliptic complex in Proposition~\ref{linearized-image} therefore allows us to solve}
$$
d_g^+ a = \frac{h}{2}\omega
$$
in the gauge $d^* a = 0$ and $a \perp \mathcal{H}_g^1$.
\rev{Elliptic estimates for the first-order operator $d^* \oplus d_g^+$ give}
$$
\|a\|_{H^{s+1}} \le C_s \|h\|_{H^s}.
$$
Putting $\eta = da$, we have $\Lambda_\omega \eta = h$ and
\begin{equation}\label{eq:eta-bound}
\|\eta\|_{H^s} \le C_s \|h\|_{H^s}.
\end{equation}
Thus $D\Phi(0)$ is an isomorphism of Banach spaces. The inverse-function theorem supplies neighborhoods $\mathcal{U}_0$ and $V_0$ of zero such that $\Phi \colon \mathcal{U}_0 \to V_0$ is a diffeomorphism.

\rev{For a normalized smooth $\rho$ with $\|\rho-1\|_{H^4}$ sufficiently small, put $F=\log\rho$. Then $e^F-1\in V_0$, so there is a unique $\psi\in\mathcal U_0$ with $\Phi(\psi)=e^F-1$. Consequently,}
$$
(\omega + \psi)^2 = e^F \omega^2, \qquad \omega + \psi > 0.
$$
\begin{revision}
To prove smoothness, work with the coexact primitive $a=d^*G\psi\in H^5$. In local coordinates $a$ satisfies
\[
 d^*a=0,\qquad \pi_J^- da=0,\qquad
 \frac{(\omega+da)^2}{\omega^2}=\rho.
\]
\textcolor{black}{Let $L_a$ be the linearization of this system. For a one-form $b$,
\[
 L_a b=\left(d^*b,\ \pi_J^-db,\
 \frac{2(\omega+da)\wedge db}{\omega^2}\right).
\]
With the real symbol convention $\partial_j\mapsto\xi_j$, its principal symbol is
\[
 \sigma_\xi(L_a)b=\left(-\langle\xi,b\rangle_g,\
 \pi_J^-(\xi\wedge b),\
 \frac{2(\omega+da)\wedge\xi\wedge b}{\omega^2}\right).
\]
This maps $T_x^*M$ to $\mathbb R\oplus\Lambda_{J,x}^-\oplus\mathbb R$, both of dimension four.} At $a=0$\textcolor{black}{,} the \textcolor{black}{identification $(c,\beta,t)\mapsto(c,\beta+(t/2)\omega)$ gives the} symbol of $d^*\oplus d_g^+$\textcolor{black}{. For $\xi\ne0$,} the \textcolor{black}{relation $\pi_g^+(\xi\wedge b)=0$ forces $b$ to be a multiple} of \textcolor{black}{$\xi$, and $\langle\xi,b\rangle_g=0$ then gives $b=0$. Thus} the \textcolor{black}{symbol is invertible}. Ellipticity persists for $\|da\|_{C^0}$ sufficiently small, so shrink $\mathcal U_0$ accordingly.

\textcolor{black}{Fix} $0<\alpha<1$\textcolor{black}{.} In dimension four, $H^5$ embeds into $C^{2,\alpha}$\textcolor{black}{, so $a\in C^{2,\alpha}$. Suppose that $a\in C^{k,\alpha}$} for \textcolor{black}{some integer $k\ge2$}. Differentiating in a coordinate direction\textcolor{black}{, in smooth local frames, gives
\[
 L_a(\partial_j a)=f_j.
\]
The coefficients of $L_a$ and} the \textcolor{black}{right-hand side $f_j$ belong to $C^{k-1,\alpha}$}. \textcolor{black}{Indeed, the} coefficients depend smoothly on $x$ and $da$\textcolor{black}{, while $f_j$ involves} derivatives of \textcolor{black}{the smooth functions and tensors} $\rho$, $\omega$, $J$ and $g$, together with $a$ \textcolor{black}{and its} first derivatives. Interior \textcolor{black}{Schauder regularity for linear} elliptic \textcolor{black}{systems \cite{TaylorNE} gives $\partial_j a\in C^{k,\alpha}_{\mathrm{loc}}$, and hence $a\in C^{k+1,\alpha}_{\mathrm{loc}}$}. Iteration yields $a\in C^\infty$, hence $\psi\in\mathcal E_J$. Its canonical potential is $T_J\psi$. Uniqueness throughout the class follows from Proposition~\ref{prop:uniqueness-exact-J-invariant}.
\end{revision}
\end{proof}

\begin{revision}
    \begin{proof}[Proof of Theorem \ref{thm:main}]
        Corollary \ref{exact lemma} \textcolor{black}{gives the unique representation by a canonical potential.} Proposition \ref{prop:uniqueness-exact-J-invariant} \textcolor{black}{gives uniqueness of the compatible form}, and \textcolor{black}{uniqueness of the representation then gives uniqueness of the potential. Finally,} Theorem \ref{local-volume} with $\rho:=e^F$ \textcolor{black}{gives smooth local existence under the stated four-dimensional hypothesis}.
    \end{proof}
\end{revision}

\rev{We now construct a closed four-torus with $h_J^-<b^+-1$ and volume forms that no compatible form in $[\omega]$ realizes. These volume forms have the same total volume as the original volume form and can be chosen arbitrarily close to it. \textcolor{black}{This shows that the hypothesis in the local existence theorem cannot be omitted in general.}}

Let
$$
M = T^4 = (\mathbb{R}/2\pi\mathbb{Z})^4,
$$
with periodic coordinates \(x, y, z, w\). Put \(e^1 = dx\), \(e^2 = dy\), \(e^3 = dz\), \(e^4 = dw\), and \(e^{ij} = e^i \wedge e^j\). Define
\begin{equation}\label{eq:omega-T4}
\omega = e^{12} + e^{34}.
\end{equation}
Then \(d\omega = 0\) and \(\omega^2/2 = e^{1234}\) is positive.

Fix \(0 < \varepsilon < 1\) and set
\begin{equation}\label{eq:u-def}
u(x) = \varepsilon \cos x.
\end{equation}
Define
$$
\kappa(x) = \frac{1}{2}\operatorname{arctanh} u(x), \quad c = \cosh\kappa, \quad s = \sinh\kappa.
$$
Use the global coframe
\begin{equation}\label{eq:global-coframe}
E^1 = ce^1 + se^4, \quad E^2 = ce^2 + se^3, \quad E^3 = se^2 + ce^3, \quad E^4 = se^1 + ce^4.
\end{equation}
The identity \(c^2 - s^2 = 1\) gives
\begin{equation}\label{eq:omega-invariant}
E^{12} + E^{34} = e^{12} + e^{34} = \omega.
\end{equation}
Let \(g = \sum_{j=1}^4 E^j \otimes E^j\), and define \(J\) on the dual frame by
$$
JE_1 = E_2, \quad JE_2 = -E_1, \quad JE_3 = E_4, \quad JE_4 = -E_3.
$$
Then \(g(X, Y) = \omega(X, JY)\) and \((M, \omega, J, g)\) is a closed almost K\"{a}hler manifold. For \(\varepsilon \neq 0\) it is nonintegrable: at \(x = \pi/2\), where \(\kappa = 0\), let
$$
\theta^1 = E^1 + \sqrt{-1}E^2, \quad \theta^2 = E^3 + \sqrt{-1}E^4.
$$
\rev{Differentiation at this point gives}
$$
\rev{d\theta^1 = \sqrt{-1}\kappa'(x)e^1 \wedge \overline{\theta^2},}
$$
\rev{whose $(0,2)$ component is $\tfrac{\sqrt{-1}\kappa'(x)}{2}\overline{\theta^1}\wedge\overline{\theta^2}$. Since $\kappa'(\pi/2)=-\varepsilon/2$, the structure is nonintegrable. As $\varepsilon\to0$, it converges smoothly to the flat K\"ahler structure.}

\textcolor{black}{To exhibit a harmonic self-dual form with nonconstant trace, define}
\begin{equation}\label{eq:eta-def}
\eta = e^{13} - e^{24} + 2u e^{12}.
\end{equation}
Since \(du\) is proportional to \(e^1\),
$$
d\eta = 2du \wedge e^{12} = 0.
$$
Set \(\sigma = \eta - u\omega\). Then
$$
\sigma = e^{13} - e^{24} + u(e^{12} - e^{34}).
$$
\rev{Expanding the transformed coframe gives}
$$
E^{13} - E^{24} = \cosh(2\kappa) (e^{13} - e^{24}) + \sinh(2\kappa) (e^{12} - e^{34}).
$$
Because \(\tanh(2\kappa) = u\),
\begin{equation}\label{eq:sigma-express}
\sigma = \frac{E^{13} - E^{24}}{\cosh(2\kappa)}.
\end{equation}
The form \(E^{13} - E^{24}\) is \(J\)-anti-invariant, so \(\sigma \in \Omega_J^-\). In four dimensions both \(\omega\) and \(\sigma\) are self-dual, hence \(\eta = u\omega + \sigma\) is self-dual. Since \(d\eta = 0\),
$$
d^*\eta = - * d * \eta = 0,
$$
so
\begin{equation}\label{eq:eta-harmonic}
\eta \in \mathcal{H}_g^+.
\end{equation}
Because \(\sigma \perp \omega\) pointwise and \(|\omega|_g^2 = 2\),
\begin{equation}\label{eq:Tomega-eta}
T_\omega(\eta) = \langle \eta, \omega \rangle_g = 2u = 2\varepsilon \cos x.
\end{equation}
\rev{This nonconstant trace gives \(h_J^- < b^+(T^4) - 1\); here \(b^+(T^4) = 3\).}

Let \(\psi \in \mathcal{E}_J\) be arbitrary, so \(\psi = da\) and \(J\psi = \psi\).
Since \(\sigma\) is \(J\)-anti-invariant and \(\psi\) is \(J\)-invariant, \(\langle\sigma, \psi\rangle_g = 0\). Thus
$$
\begin{aligned}
\int_M u\Lambda_\omega\psi \, dV_g
&= \int_M \langle u\omega, \psi\rangle \, dV_g \\
&= \int_M \langle \eta - \sigma, \psi\rangle \, dV_g \\
&= \int_M \langle \eta, da\rangle \, dV_g \\
&= \int_M \langle d^*\eta, a\rangle \, dV_g = 0.
\end{aligned}
$$
Therefore
\begin{equation}\label{eq:orthogonality-u}
\int_M u\Lambda_\omega\psi \, dV_g = 0 \quad \forall \psi \in \mathcal{E}_J.
\end{equation}
\rev{Exactness also gives \(\int_M \Lambda_\omega\psi \, dV_g = 0\). Fix \(C > \|u\|_{L^\infty}\) and put \(q = C + u > 0\). Then}
\begin{equation}\label{eq:orthogonality-q}
\int_M q\Lambda_\omega\psi \, dV_g = 0 \quad \forall \psi \in \mathcal{E}_J.
\end{equation}

\textcolor{black}{To turn this orthogonality into a restriction on the volume, suppose} a cohomologous $J$-compatible form $\Omega = \omega + \psi > 0$ solves
\begin{equation}\label{eq:volume-ODE}
\Omega^2 = e^F\omega^2.
\end{equation}
At each point diagonalize $\Omega$ relative to $\omega$. Let $\lambda_1, \lambda_2 > 0$ be the eigenvalues. Then $e^F = \lambda_1\lambda_2$, and the inequality $\log t \le t - 1$ gives
$$
F = \log \lambda_1 + \log \lambda_2 \le (\lambda_1 - 1) + (\lambda_2 - 1) = \Lambda_\omega\psi.
$$
Thus
\begin{equation}\label{eq:F-leq-Lambda}
F \le \Lambda_\omega\psi.
\end{equation}
Multiplying by $q > 0$, integrating, and using (\ref{eq:orthogonality-q}), we obtain the necessary inequality
\begin{equation}\label{eq:necessary-inequality}
\int_M qF \, dV_g \le 0.
\end{equation}
\rev{This restriction is geometric and does not depend on how the scalar potential and the anti-invariant two-form are constructed.}

\textcolor{black}{We construct smooth functions that satisfy the volume normalization but contradict the preceding inequality.} Let \(V = \int_M dV_g\). For \(t > 0\) define
\begin{equation}\label{eq:Ft-def}
F_t = t u - c(t), \qquad c(t) = \log \left( \frac{1}{V} \int_M e^{t u} dV_g \right).
\end{equation}
By definition \(\int_M e^{F_t} dV_g = V\), equivalently \(\int_M e^{F_t} \omega^2 = \int_M \omega^2\). Thus every \(F_t\) is normalized and \(F_t \to 0\) smoothly.

Because \(u = \varepsilon \cos x\) has mean zero and is nonzero,
$$
\frac{1}{V} \int_M e^{t u} dV_g = 1 + \frac{t^2}{2V} \int_M u^2 dV_g + O(t^3),
$$
so
\begin{equation}\label{eq:c-expansion}
c(t) = \frac{t^2}{2V} \int_M u^2 dV_g + O(t^3).
\end{equation}
\rev{Substitution into the weighted integral gives}
$$
\begin{aligned}
\int_M q F_t dV_g &= \int_M (C + u)(t u - c(t)) dV_g \\
&= t \int_M u^2 dV_g - CV c(t) \\
&= t \int_M u^2 dV_g + O(t^2).
\end{aligned}
$$
For sufficiently small \(t > 0\) the right-hand side is strictly positive, contrary to (\ref{eq:necessary-inequality}).
\rev{Thus no compatible form cohomologous to $\omega$ can satisfy $\Omega^2=e^{F_t}\omega^2$. In particular, no potential for $\widehat{\mathcal D}_J^+$ produces this volume.}

\begin{theo}\label{theo:nonexistence}
\rev{There is a closed almost K\"ahler four-torus $(T^4,\omega,J,g)$ with nonintegrable $J$ and $h_J^-<b^+-1$, and a family of smooth functions $F_t\to0$ in $C^\infty$, such that}
$$
\int_{T^4} e^{F_t} \omega^2 = \int_{T^4} \omega^2
$$
\rev{and no $J$-compatible symplectic form $\Omega$ in $[\omega]$ satisfies $\Omega^2=e^{F_t}\omega^2$ for sufficiently small $t>0$.}
\end{theo}
\begin{revision}
The same harmonic form $\eta=u\omega+\sigma$ shows that the full-domain scalar kernel contains the nonconstant mean-zero function $u$: take $-\eta=-u\omega-\sigma$ in the kernel description above. This explains why uniqueness of the form and uniqueness of a full-domain scalar are distinct even for $F=0$.
\end{revision}
\begin{rem}\label{rem:open-dense}
For any closed symplectic 4-manifold \((M, \omega)\), the set of \(\omega\)-compatible almost complex structures \(J\) with \(h_J^- = 0\) is open and dense in the \(C^\infty\)-topology (see the main result of Tan, Wang and Zhou~\cite{TWZ}).
\end{rem}

 \section{An elliptic system for the $\mathcal{D}_J^+$ operator}\setcounter{equation}{0}\label{elliptical system}
  \textcolor{black}{We construct a family of trace potentials and compare the normalized endpoint potential with the one used in the a priori estimate below.} Suppose that $(M,\omega,J,g)$ is a closed almost K\"{a}hler manifold of dimension $2n$.
  If $f\in\mathcal{F}_J$ and
$\omega_1=\omega+\mathcal{D}_J^+(f)>0$, then $\omega_1$ is a $J$-compatible symplectic form.
   Let
   \begin{equation*}
\omega_t=t\omega_1+(1-t)\omega,\quad0\leq t\leq1.
\end{equation*}
   \rev{The family $\omega_t$ consists of $J$-compatible symplectic forms in the same symplectic class $[\omega]$, with associated almost K\"{a}hler metrics $g_t(\cdot,\cdot)=\omega_t(\cdot,J\cdot)$.}
\begin{revision}
The primitive-form decomposition expresses the trace equation below in terms of the symplectic differential operators used in primitive cohomology. For $0\le k\le n$, let
$\mathcal P_t^k=\ker\Lambda_{\omega_t}\cap\Omega^k_{\mathbb R}(M)$,
and put $\mathcal P_t^{-1}=\mathcal P_t^{n+1}=\{0\}$.
The Lefschetz decomposition defines
$\partial_{\pm,t}:\mathcal P_t^k\longrightarrow\mathcal P_t^{k\pm1}$ by
\begin{equation}\label{eq:primitive-differential}
 dB=\partial_{+,t}B+\omega_t\wedge\partial_{-,t}B,
 \qquad B\in\mathcal P_t^k,
\end{equation}
with $\partial_{+,t}B=0$ for $k=n$
\cite[Lemma~2.4]{TY} and \cite[Definition~2.4]{TY2}.
Since $\Lambda_{\omega_t}(\omega_t\wedge C)=(n-k+1)C$
for $C\in\mathcal P_t^{k-1}$, contraction gives
\begin{equation*}
 \partial_{-,t}B=\frac{1}{n-k+1}\Lambda_{\omega_t}dB.
\end{equation*}
Every real $J$-anti-invariant two-form
$\alpha=\alpha^{2,0}+\alpha^{0,2}$ is primitive. In particular,
\begin{equation}\label{eq:anti-invariant-primitive}
\begin{aligned}
 d\alpha&=\partial_{+,t}\alpha+
 \omega_t\wedge\partial_{-,t}\alpha,\\
 \partial_{-,t}\alpha&=\frac{1}{n-1}\Lambda_{\omega_t}d\alpha,
 \qquad \omega_t^{n-2}\wedge\partial_{+,t}\alpha=0.
\end{aligned}
\end{equation}
When $n=2$, the last identity means $\partial_{+,t}\alpha=0$.
This decomposition also enters the construction of the twisted Aubin functionals in \cite{WZZa}.
For every real one-form $b$, formula~\eqref{eq:primitive-differential} gives
\begin{equation}\label{eq:primitive-one-form-trace}
 \partial_{-,t}b=\frac1n\Lambda_{\omega_t}db,
 \qquad
 \partial_{+,t}b=db-\frac{\Lambda_{\omega_t}db}{n}\omega_t.
\end{equation}

We use $\Delta_{g_t}=d_t^*d$ on functions, where $d_t^*$ is the codifferential for $g_t$. For every real smooth $u$,
\begin{equation}\label{eq:trace-laplacian}
 \Lambda_{\omega_t}dJdu=-\Delta_{g_t} u.
\end{equation}
Indeed, \textcolor{black}{integrating} the left-hand side \textcolor{black}{against} a test function $v$ \textcolor{black}{and integrating by parts} gives
$\langle d_t^*(v\omega_t),Jdu\rangle_{g_t}
=-\langle dv,du\rangle_{g_t}$.

\begin{theo}\label{elliptical system theo}
Set $\psi=\omega_1-\omega$. For each $t$, let $f_t$ be the mean-zero solution, with respect to $g_t$, of
\begin{equation}\label{def of ft}
 -\Delta_{g_t} f_t=\Lambda_{\omega_t}\psi
 =n\,\frac{\omega_t^{n-1}\wedge\psi}{\omega_t^n}.
\end{equation}
There is a unique $g_t$-coexact one-form $a_t$ such that
\[
 \psi=dJdf_t+da_t .
\]
It satisfies the elliptic system
\begin{equation}\label{TWY5.4}
 d_t^*a_t=0,\qquad d_J^-a_t=-d_J^-Jdf_t,\qquad
 \Lambda_{\omega_t}da_t=0.
\end{equation}
\end{theo}
\begin{proof}
Since $\psi$ is exact, $\int_M\psi\wedge\omega_t^{n-1}=0$. This \textcolor{black}{is the compatibility condition for the scalar Laplace equation, so} \eqref{def of ft} \textcolor{black}{is solvable} with the stated normalization. If $G_t$ denotes the Green operator on two-forms for $g_t$, put
\[
 b_t=d_t^*G_t\psi,\qquad a_t=b_t-Jdf_t .
\]
Then $db_t=\psi$. Both terms defining $a_t$ are coexact, since $Jdf_t=-d_t^*(f_t\omega_t)$. \textcolor{black}{Any two such coexact one-forms differ by} a closed coexact one-form\textcolor{black}{, which must vanish}. \textcolor{black}{This proves uniqueness.} Projecting $da_t=\psi-dJdf_t$ onto $\Lambda_J^-$ gives the anti-invariant equation; taking the trace and using \eqref{eq:trace-laplacian} establishes primitivity. In terms of the symplectic differential operators, equations \eqref{def of ft} and \eqref{eq:primitive-one-form-trace} read
\begin{equation}\label{eq:primitive-trace-system}
 \partial_{-,t}b_t=\partial_{-,t}(Jdf_t)
 =-\frac1n\Delta_{g_t} f_t,
 \qquad \partial_{-,t}a_t=0.
\end{equation}
Consequently, $da_t=\partial_{+,t}a_t$ and
$\pi_J^-\partial_{+,t}a_t=-d_J^-Jdf_t$.
This gives the primitive-form expression of the same elliptic system.

The first-order operator in \eqref{TWY5.4} has injective symbol. If its symbol annihilates a one-form $v$ at $\xi\ne0$, then $v\perp\xi$ and $\xi\wedge v$ is invariant and primitive. If $v\ne0$, invariance forces the plane spanned by $\xi,v$ to be a complex line; the area form of such a plane has nonzero trace. \textcolor{black}{This contradicts primitivity, so} $v=0$. For $n>2$ this system is overdetermined.
\end{proof}
\end{revision}

  \vskip 6pt

\begin{revision}
The kernel of the homogeneous system \eqref{TWY5.4} on all one-forms is $\mathcal H^1_{g_t}$. Indeed, if $v$ lies in this kernel, then $dv$ is primitive and invariant. The Hodge star on this component gives
\[
 \|dv\|^2_{L^2(g_t)}
 =-\frac{1}{(n-2)!}\int_M dv\wedge dv\wedge\omega_t^{n-2}=0
\]
by Stokes' theorem. Thus $v$ is closed and coclosed. This identifies the harmonic freedom removed by choosing $a_t$ coexact.

At $t=1$, equation \eqref{def of ft} becomes
\begin{equation}\label{lap trace}
 \Delta_{g_1} f_1=\tfrac12\operatorname{tr}_{g_1}g-n .
\end{equation}
Write $p=f_1-\sup_M f_1$. Tosatti, Weinkove and Yau use a potential $\varphi$ of supremum zero. In our sign convention, their defining equation is
\[
 -\frac{1}{2n}\Delta_{g_1}\varphi
 =1-\frac{\omega_1^{n-1}\wedge\omega}{\omega_1^n}.
\]
It follows from \eqref{lap trace} that $\varphi=2p$. Applying their estimate with exponent $\alpha/2$ and using $\omega^n=n!\,dV_g$ gives the following consequence of \cite{TWY}.
\end{revision}
\begin{revision}
\begin{theo}\label{first main result}
Let $\alpha>0$, and let $\omega_1$ be a $J$-compatible form in $[\omega]$ satisfying $\omega_1^n=e^F\omega^n$ for a smooth $F$ with $\int_Me^F\omega^n=\int_M\omega^n$. Define $p=f_1-\sup_M f_1$ by \eqref{def of ft}. Then $\omega_1$ satisfies $C^\infty$ a priori bounds depending only on $\omega,J,F,\alpha$ and \textcolor{black}{the $\alpha$-integral (cf.~Tian~\cite{Tian})}
\[
 I_\alpha(p)=\int_M e^{-\alpha p}\,dV_g .
\]
\end{theo}
\end{revision}

\begin{revision}
The trace potential $f_t$ in \eqref{def of ft} is determined by the pair of forms and the metric $g_t$. The canonical potential $f=T_J(\omega_1-\omega)$ is determined by the elliptic representation for the original metric $g$.
\end{revision}

\section{Further Remarks and Questions}\setcounter{equation}{0}

\begin{revision}
For fixed $J$, the normalized potential space is
\begin{equation}\label{eq:positive-canonical-space}
 \mathcal H(\omega,J)=
 \{f\in\mathcal F_J:\omega+\mathcal D_J^+(f)>0\}
 =\mathcal F_J^+.
\end{equation}
It parametrizes the compatible forms in $[\omega]$ by
$f\mapsto\omega+\mathcal D_J^+(f)$. Since
$\mathcal D_J^+(0)=0$, it contains $0$. The space is convex because the operator is linear and positivity is preserved under convex combinations. If $0\ne\psi\in\mathcal E_J$, then
$tT_J\psi\in\mathcal H(\omega,J)$ for every sufficiently small real $t$, because \textcolor{black}{compactness and positivity of the background form give} $\omega+t\psi>0$. Thus this space contains nonzero potentials exactly when $\mathcal E_J\ne\{0\}$. Nonzero potentials exist whenever $J$ is integrable and also in real dimension four: in the latter case \eqref{eq:scalar-domain-splitting} expresses $C^\infty_0(M)$ as the sum of $\mathcal F_J$ and a finite-dimensional kernel. \textcolor{black}{The space of mean-zero smooth functions is infinite-dimensional, so the canonical summand cannot vanish.}

In real dimension four, the full mean-zero potential space is
\begin{equation}\label{eq:positive-full-space}
\begin{aligned}
 \widehat{\mathcal H}(\omega,J)
 &=\{u\in C^\infty_0(M):\omega+\widehat{\mathcal D}_J^+(u)>0\}\\
 &=\mathcal H(\omega,J)+
 \ker\bigl(\widehat{\mathcal D}_J^+|_{C^\infty_0(M)}\bigr).
\end{aligned}
\end{equation}
The decomposition in \eqref{eq:scalar-domain-splitting} shows that each element has a unique canonical representative in $\mathcal H(\omega,J)$. These two spaces give distinct choices for studying the geometry of almost K\"ahler potentials when the scalar kernel is nonzero. If constants are adjoined to the canonical potential space, the quotient of the enlarged space by constants is identified with $\mathcal F_J^+$ by taking the mean-zero representative.

In the K\"ahler case the Mabuchi metric pairs variations by integrating their product against the varying volume form \cite{Mabu,Se,DonSym}. On $\mathcal F_J^+$ the analogous pairing is
\[
 (u,v)_f=\int_M uv\,\frac{\omega_f^n}{n!},\qquad u,v\in\mathcal F_J .
\]
For the usual Mabuchi metric on the full K\"ahler potential space \cite{DonSym,Se}, weak geodesics exist \cite{Chen}, while some smooth endpoints cannot be joined by a smooth geodesic \cite{Lem,DL}. For a nonintegrable $J$, can one describe the geodesics of this pairing by an equation on the strip $S=\{s\in\mathbb C:0<\operatorname{Im}s<1\}$ times $M$? Such a description would require an extension of the canonical operator to the product, boundary conditions for its elliptic inverse, and a derivation of the geodesic equation. The pullback of $\omega$ alone is degenerate in the strip direction.
\end{revision}

\rev{\textcolor{black}{As in the K\"{a}hler case} \cite{DonK}\textcolor{black}{, we ask whether} a fixed symplectic class on a compact manifold \textcolor{black}{admits each of} the \textcolor{black}{following} four kinds of special almost K\"{a}hler metrics.}

     {\bf 1)} Extremal almost K\"{a}hler metrics

     \textcolor{black}{For} a \textcolor{black}{fixed symplectic form $\omega$}, \textcolor{black}{let $J$ vary over all smooth $\omega$-compatible almost-complex structures. The associated} almost K\"{a}hler \textcolor{black}{metric is called extremal if $J$ is a} critical \textcolor{black}{point} of the Calabi functional \textcolor{black}{\cite{KL}}
     \begin{equation*}
        J\mapsto\int_{M}R(J)^2\frac{\omega^n}{n!},
     \end{equation*}
     where $R(J)$ is the Hermitian scalar curvature \textcolor{black}{of $(\omega,J,g)$. These variations differ from those in $\mathcal H(\omega,J)$, where $J$ is fixed and} the \textcolor{black}{compatible form varies in $[\omega]$}.
\rev{Two questions concern the uniqueness of extremal almost K\"{a}hler metrics (see \cite{BB} for the K\"{a}hler case) and their relation to extremal K\"{a}hler metrics.}
         In the toric case, the existence of an extremal K\"{a}hler metric is conjecturally equivalent to the existence of a nonintegrable extremal almost K\"{a}hler metric \cite{DonS} (see also \cite{ACG}).

    {\bf 2)} Constant Hermitian scalar curvature almost K\"{a}hler metrics

\rev{These are precisely the metrics with constant Hermitian scalar curvature $R$; all are extremal.}

    {\bf 3)} Hermite-Einstein almost K\"{a}hler metrics with $c_1=0$, $c_1<0$ and $c_1>0$

     An almost K\"{a}hler metric $(\omega,J,g)$ is called Hermite-Einstein (HEAK for short) if the Hermite-Ricci form $\rho$
              is a (constant) multiple of the symplectic form $\omega$, i.e.
                \begin{equation*}
\rho=\frac{R}{2n}\omega,
\end{equation*}
      so the Hermitian scalar curvature $R$ is constant (cf. \cite{Lej2}).
\rev{The form $\rho$ represents $2\pi c_1(M,J)$ in de Rham cohomology $H^2(M,\mathbb{R})$, where $c_1(M,J)$ is the first (real) Chern class.}
       If $\omega$ and $\tilde{\omega}$ are symplectic forms compatible with the same almost-complex structure
\rev{$J$ and satisfy $\tilde{\omega}^n=e^F\omega^n$ for a real function $F$, the curvature transformation law \cite{ABL} gives}
  \begin{equation*}
    \tilde{\rho}=\rho-\frac{1}{2}dJdF,
  \end{equation*}
  where $\tilde{\rho}$ (resp. $\rho$) is the Hermite-Ricci form of $(\tilde{\omega},J,\tilde{g})$ (resp. $(\omega,J,g)$).
\begin{revision}
Suppose that $\tilde\omega$ lies in $[\omega]$, and write $f=T_J(\tilde\omega-\omega)$ for its canonical potential. Then $F=\log(\tilde\omega^n/\omega^n)=\log M_J(f)$. \textcolor{black}{The associated volume forms satisfy}
\[
 \textcolor{black}{e^F=\frac{\tilde\omega^n}{\omega^n}
 =\frac{dV_{\tilde g}}{dV_g}
 =\sqrt{\det(g^{-1}\tilde g)}.}
\]
If $\tilde\rho=\lambda\tilde\omega$, cohomology gives the necessary condition
\[
 2\pi c_1(M,J)=\lambda[\omega].
\]
Under this proportionality, define $h$ by $\int_M h\,dV_g=0$ and
\[
 -\Delta_g h=\textcolor{black}{2}\Lambda_\omega(\rho-\lambda\omega).
\]
The right-hand side has zero mean since $\rho-\lambda\omega$ is exact. Let $f_0$ be the trace potential in \eqref{def of ft} for the pair $\omega,\tilde\omega$. Taking the $\omega$-trace of the full curvature identity
\[
 \rho-\tfrac12 dJdF=\lambda\tilde\omega
\]
and using \eqref{eq:trace-laplacian} gives
\[
 \Delta_g(F-\textcolor{black}{h}+2\lambda f_0)=0 .
\]
\textcolor{black}{The function in parentheses is therefore constant. Writing this as $F+2\lambda f_0=\textcolor{black}{h}+C$,} a HEAK metric satisfies the scalar equation
\begin{equation}\label{exis equ}
 \log M_J(f)=\textcolor{black}{F=}\textcolor{black}{h}-2\lambda f_0+C
\end{equation}
for a constant $C$. Here the trace of $\tilde\omega-\omega$ defines $f_0$, while $M_J$ takes the canonical potential $f$ as its argument.

Conversely, if a compatible $\tilde\omega\in[\omega]$ satisfies \eqref{exis equ}, the form $\gamma=\tilde\rho-\lambda\tilde\omega$ is exact and primitive. If $\tilde\rho$ is also $J$-invariant, so is $\gamma$, and the primitive Hodge-star identity gives
\[
 \|\gamma\|_{L^2(g)}^2
 =-\frac{1}{(n-2)!}\int_M\gamma\wedge\gamma\wedge\omega^{n-2}=0 .
\]
Hence $\tilde\omega$ is HEAK under this additional condition. The Hermitian Ricci form of an almost K\"ahler metric can have an anti-invariant component \cite{ABL}. Existence in a proportional class therefore requires controlling this component as well as solving the traced equation.

For $\lambda\ne0$, \textcolor{black}{define} the rescaled trace potential \textcolor{black}{by
$\varphi=2|\lambda|f_0-(\lambda/|\lambda|)C$}; for $\lambda=0$, \textcolor{black}{set $\varphi=2f_0$}.
These potentials are admissible when they arise from a canonical potential
$f\in\mathcal F_J^+$, with $C\in\mathbb R$ in the nonzero case.
The form $\tilde\omega=\omega+\mathcal D_J^+(f)$ is uniquely determined by
\textcolor{black}{$\varphi$}. Indeed, two such forms with the same \textcolor{black}{rescaled trace} potential
have the same $\omega$-trace.
Their difference is exact, $J$-invariant and primitive, so the preceding
Hodge-star identity and Stokes' theorem show that it vanishes.
In \eqref{eq:heak-sign-cases}, we write \textcolor{black}{$M_J(\varphi)$} for the volume ratio
\textcolor{black}{$\tilde\omega^n/\omega^n$} of the \textcolor{black}{form} associated \textcolor{black}{with $\varphi$}.
\textcolor{black}{Thus $M_J(\varphi)=M_J(f)$ for} the \textcolor{black}{corresponding} canonical potential $f$\textcolor{black}{;
this use of $M_J$ is confined to} \eqref{eq:heak-sign-cases}\textcolor{black}{.
The rescaled trace potential $\varphi$} need not have zero mean when $\lambda\ne0$.
Under the proportionality $2\pi c_1(M,J)=\lambda[\omega]$,
the existence of a HEAK metric in $[\omega]$ is therefore equivalent to solving
\begin{equation}\label{eq:heak-sign-cases}
 \begin{cases}
  \log M_J(\textcolor{black}{\varphi})=\textcolor{black}{\varphi}+\textcolor{black}{h}, & \text{if }c_1(M,J)<0,\\
  \log M_J(\textcolor{black}{\varphi})=\textcolor{black}{h}+C, & \text{if }c_1(M,J)=0,\\
  \log M_J(\textcolor{black}{\varphi})=-\textcolor{black}{\varphi}+\textcolor{black}{h}, & \text{if }c_1(M,J)>0,
 \end{cases}
\end{equation}
for admissible \textcolor{black}{$\varphi$}, subject to the condition that the corresponding
Hermite-Ricci form $\tilde\rho$ is $J$-invariant.
Under the stated proportionality, the three signs of $c_1(M,J)$ correspond
to $\lambda<0$, $\lambda=0$ and $\lambda>0$, respectively.
In terms of the mean-zero trace potential $f_0$, the total volume determines
the constant:
\begin{equation}\label{eq:heak-normalization}
 C=\log\frac{\displaystyle\int_M\omega^n}
 {\displaystyle\int_M e^{\textcolor{black}{h}-2\lambda f_0}\omega^n}.
\end{equation}
As in the K\"ahler case, if $c_1(M,J)=0$, the scalar equation is equivalent
to \eqref{generalized MA} for the canonical potential $f$, with the fixed normalized function
\[
 F=\textcolor{black}{h}+\log\bigl(\int_M\omega^n/\int_M e^{\textcolor{black}{h}}\omega^n\bigr).
\]
If $J$ is integrable, then $f_0=f$ and the Ricci form is automatically
$J$-invariant. For $\lambda\ne0$, scaling both $\omega$ and $\tilde\omega$
by $|\lambda|$ makes the Einstein constant of a HEAK solution equal to
$\lambda/|\lambda|$. The volume ratio is unchanged, and the \textcolor{black}{rescaled trace} potential \textcolor{black}{$\varphi$}
in \eqref{eq:heak-sign-cases} is a K\"ahler potential for the scaled pair.
The constant $C$ is absorbed into this potential when $\lambda\ne0$
and remains explicit when $\lambda=0$.
For the K\"ahler case, see \cite{Aubin1998}.
\end{revision}

     {\bf 4)} Generalized almost K\"{a}hler-Ricci solitons\par\nopagebreak

     For \textcolor{black}{the literature on} K\"{a}hler-Ricci solitons, see \cite{Zhu,WXJZ} and the references therein.
\begin{revision}
On a compact symplectic Fano manifold\textcolor{black}{, normalize so that} $[\omega]=2\pi c_1(M,J)$\textcolor{black}{. Fix} a Hamiltonian torus and an element $\xi$ of its Lie algebra. For torus-invariant compatible structures, write $s^c=\Lambda_\omega\rho=R/2$. Choose the Hamiltonian $f_\xi$ by $\iota_\xi\omega=-df_\xi$ and $\int_M f_\xi e^{-2f_\xi}\omega^n=0$. In the real convention of \cite{ABL} the modified curvature is
\[
 s^c_\xi=s^c-n-2\Delta_g f_\xi+2f_\xi-2|df_\xi|_g^2,\qquad \Delta_g=d^*d .
\]
For the action of the torus centralizer, the moment map pairs $4s^c_\xi$ with torus-invariant Hamiltonians of weighted mean zero against $e^{-2f_\xi}\omega^n$ \cite{ABL}. On the integrable locus the zero equation is the K\"ahler--Ricci soliton equation \cite[Proposition~3.2]{Inoue}. On the full space of torus-invariant compatible structures $s^c_\xi=0$ defines the generalized solitons studied in \cite{ABL}.
\end{revision}

\begin{appendices}

\section{The definition of $\mathcal{F}_J$}\label{appendices}
\renewcommand{\theequation}{A.\arabic{equation}}
\setcounter{equation}{0}
\begin{revision}
In dimension four, the equation for the $J$-anti-invariant two-form can be solved directly. In dimension $2n\ge6$ the principal symbol of $d_J^-d^*$ on this bundle has a kernel. Indeed, choose a unitary coframe at a point and take $\xi=dx_1$ and $\beta=\operatorname{Re}(dz_2\wedge dz_3)$ in the associated real coframe. Then $\beta\ne0$ but $\iota_{\xi^\sharp}\beta=0$, so the symbol annihilates $\beta$. This calculation does not determine the range of the operator. To represent a given exact invariant two-form we shall instead solve an elliptic equation for a one-form.
\end{revision}

  Recall the orthogonal decomposition
\begin{equation}\label{eq:orth-decomp}
\Lambda^2 T^*M
=
\mathbb{R}\omega
\oplus
\Lambda_{J,0}^+
\oplus
\Lambda_J^-,
\end{equation}
where
$$
\Lambda_{J,0}^+
=
\{\alpha\in\Lambda_J^+:\Lambda_\omega\alpha=0\}
$$
is the space of primitive $J$-invariant two-forms.

  Define
$$
E=\mathbb{R}\omega\oplus\Lambda_J^-.
$$
Let
$$
Q:\Lambda^2T^*M\longrightarrow E
$$
be the orthogonal projection.
  Since $|\omega|_g^2=n$, the projection has the explicit formula
\begin{equation}\label{eq:Q-formula}
Q\alpha
=
\frac{1}{n}(\Lambda_\omega\alpha)\omega
+
\pi_J^-\alpha.
\end{equation}
\rev{As an orthogonal projection, $Q$ satisfies}
$$
Q^2=Q,\qquad Q^*=Q,
$$
and
$$
\langle Q\alpha,\alpha\rangle_g=\|Q\alpha\|^2.
$$
Its kernel is
\begin{equation}\label{eq:kerQ}
\ker Q=\Lambda_{J,0}^+.
\end{equation}

  \textcolor{black}{We use this projection to define an operator on one-forms:}
\begin{equation}\label{new ellip}
\mathcal{L}_J
=
d^*Qd+dd^*
:\Omega^1(M)\longrightarrow\Omega^1(M).
\end{equation}
\rev{The quadratic form of $\mathcal L_J$ is}
\begin{equation}\label{positivity}
\begin{aligned}
\langle \mathcal{L}_J \theta,\theta\rangle_g
&=\langle d^*Qd\theta,\theta\rangle_g+\langle dd^*\theta,\theta\rangle_g\\
&=\langle Qd\theta,d\theta\rangle_g+\|d^*\theta\|^2\\
&=\|Qd\theta\|^2+\|d^*\theta\|^2.
\end{aligned}
\end{equation}
\rev{This calculation uses $Q^*=Q=Q^2$ and shows that $\mathcal{L}_J$ is formally self-adjoint and nonnegative. To prove strong ellipticity, we examine the principal symbol.}

Fix a point $x\in M$ and a nonzero cotangent vector $\xi\in T_x^*M$.
The quadratic form of the principal symbol is
\begin{equation}\label{principal symbol}
\langle\sigma_\xi(\mathcal{L}_J)\theta,\theta\rangle
=
|Q(\xi\wedge \theta)|^2+|\iota_{\xi^\sharp}\theta|^2.
\end{equation}
\rev{If this expression vanishes, then}
\begin{equation}\label{expression vanishes}
Q(\xi\wedge \theta)=0,\qquad \iota_{\xi^\sharp}\theta=0.
\end{equation}
\rev{If $\xi\wedge\theta=0$, then $\theta$ is proportional to $\xi$ and the contraction condition gives $\theta=0$. Otherwise the nonzero form $\xi\wedge\theta$ is $J$-invariant.}
A decomposable nonzero $J$-invariant real two-form determines a $J$-invariant real two-plane.
Since that plane contains $\xi^\sharp$, it must be
$$
\operatorname{span}_{\mathbb{R}}\{\xi^\sharp,J\xi^\sharp\}.
$$
Consequently, $\theta=c_1\xi+c_2J\xi$ for real numbers $c_1,c_2$.
The contraction condition in (\ref{expression vanishes}) yields $c_1|\xi|^2=0$, so $c_1=0$, leaving $\theta=c_2J\xi$.
The scalar component of $Q(\xi\wedge \theta)=0$ gives
$$
0=\Lambda_\omega(\xi\wedge \theta)=c_2|\xi|^2.
$$
\rev{Thus $c_2=0$ and $\theta=0$.}
\rev{The principal symbol is positive definite for every nonzero $\xi$, and compactness gives the uniform estimate}
\begin{equation}\label{uniform}
\langle\sigma_\xi(\mathcal{L}_J)\theta,\theta\rangle
\ge c_0|\xi|^2|\theta|^2
\end{equation}
for some constant $c_0>0$. Thus $\mathcal{L}_J$ is strongly elliptic.

\rev{We next identify the kernel
}\begin{equation}\label{ker}
\ker\mathcal{L}_J=\mathcal{H}_{g}^1,
\end{equation}
\rev{where $\mathcal H_g^1$ is the space of harmonic one-forms.}

Indeed, if $\mathcal{L}_J \theta=0$, equation (\ref{positivity}) gives $Qd\theta=0$ and $d^*\theta=0$. The first condition implies that $d\theta$ is primitive and $J$-invariant.

For any real primitive $J$-invariant two-form $\gamma\in\Omega_{J,0}^+(M)$, the Hodge star satisfies
$$
*\gamma
=
-\frac{\gamma\wedge\omega^{n-2}}{(n-2)!}.
$$
Therefore
$$
\frac{\gamma\wedge\gamma\wedge\omega^{n-2}}{(n-2)!}
=
-|\gamma|^2dvol_g.
$$
Taking $\gamma=d\theta$\textcolor{black}{, we integrate} and \textcolor{black}{apply Stokes' theorem to obtain}
$$
\begin{aligned}
\int_M |d\theta|^2\,dV_g
&=-\frac{1}{(n-2)!}\int_M d\theta\wedge d\theta\wedge\omega^{n-2}\\
&=-\frac{1}{(n-2)!}\int_M d\bigl(\theta\wedge d\theta\wedge\omega^{n-2}\bigr)\\
&=0.
\end{aligned}
$$
Thus $d\theta=0$. Combined with $d^*\theta=0$, this means $\theta$ is harmonic. Conversely, every harmonic one-form lies in $\ker\mathcal{L}_J$. This proves (\ref{ker}).

\begin{prop}\label{prop:elliptic-representation}
Let $\psi$ be a smooth exact $J$-invariant two-form. There exist smooth $f$ and $\sigma\in\Omega_J^-(M)$ such that
\begin{equation}\label{elliptic-representation}
\rev{\psi=dd^*(-f\omega+\sigma),}
\end{equation}
and
$$
\rev{d_J^-d^*(-f\omega+\sigma)=0.}
$$
\end{prop}
\begin{proof}
\rev{Hodge theory assigns to the exact form $\psi$ the canonical coexact primitive}
\begin{equation}\label{eq:coexact-primitive}
a=d^*G\psi,
\end{equation}
where $G$ is the Green operator of the Hodge Laplacian on two-forms.
Because $d\psi=0$ and $\psi$ has no harmonic component, the Hodge decomposition gives
$$
da=\psi
\quad\text{and}\quad
d^*a=0.
$$
Moreover
$$
a\perp\mathcal{H}_g^1.
$$
\rev{Invariance of $\psi$ under $J$ gives}
\begin{equation}\label{eq:dJ-a-zero}
d_J^-a=\pi_J^-da=\pi_J^-\psi=0.
\end{equation}
\rev{\textcolor{black}{Since the right-hand side is orthogonal to the kernel, the Fredholm alternative for} $\mathcal{L}_J$ allows us to solve}
\begin{equation}\label{eq:LJ-b-a}
\mathcal{L}_J \theta=a,\qquad \theta\perp\mathcal{H}_g^1.
\end{equation}
Thus
\begin{equation}\label{eq:LJ-expanded}
d^*Qd\theta+dd^*\theta=a.
\end{equation}
\rev{Applying $d^*$ gives $d^*dd^*\theta=0$. The harmonic function $d^*\theta$ has zero mean by adjunction with the constant function $1$, so it vanishes.}

Equation \eqref{eq:LJ-expanded} reduces to
\begin{equation}\label{eq:dQdb-a}
d^*Qd\theta=a.
\end{equation}
\rev{Define the scalar function and anti-invariant form by
}\begin{equation}\label{eq:f-sigma-def}
\rev{f_\psi=-\frac{1}{n}\Lambda_\omega d\theta,\qquad}
\sigma_\psi=\pi_J^-d\theta.
\end{equation}
\rev{Thus $Qd\theta=-f_\psi\omega+\sigma_\psi$, and \eqref{eq:dQdb-a} reads}
$$
\rev{d^*(-f_\psi\omega+\sigma_\psi)=a.}
$$
\rev{Applying $d$ gives $dd^*(-f_\psi\omega+\sigma_\psi)=da=\psi$. By \eqref{eq:dJ-a-zero},}
$$
\rev{d_J^-d^*(-f_\psi\omega+\sigma_\psi)=0.}
$$
This proves both required identities.

\end{proof}

Define
$$
\mathcal{E}_J=\{\psi\in\Omega_J^+(M):\psi\ \text{is exact}\}.
$$
\rev{The elliptic representation defines a linear map}
$$
T_J\colon\mathcal{E}_J\longrightarrow C^\infty(M),\qquad
T_J(\psi)=f_\psi,
$$
where
$$
\rev{f_\psi=-\frac{1}{n}\Lambda_\omega d\theta_\psi}
$$
and
$$
\theta_\psi
=
\left(
\mathcal{L}_J\big|_{(\mathcal{H}_g^1)^\perp}
\right)^{-1}
d^*G\psi.
$$
\rev{To prove that $T_J$ is injective, suppose $T_J(\psi)=0$. The representation equation then becomes}
$$
\psi=dd^*\sigma_\psi.
$$
Set $a=d^*\sigma_\psi$. Then
$$
da=\psi.
$$
Since $\sigma_\psi$ is $J$-anti-invariant while $\psi$ is $J$-invariant,
$$
\|a\|^2
=
\langle d^*\sigma_\psi,a\rangle
=
\langle \sigma_\psi,da\rangle
=
\langle \sigma_\psi,\psi\rangle
=0.
$$
\rev{Hence $a=0$ and $\psi=da=0$, proving}
\begin{equation}\label{eq:kerTJ}
\ker T_J=\{0\}.
\end{equation}
The function $f_\psi$ also has zero mean. Indeed, since
$$
*\omega=\frac{\omega^{n-1}}{(n-1)!},
$$
the condition $d\omega=0$ implies $d^*\omega=0$. Therefore
$$
\int_M f_\psi\,dvol_g
=
\rev{-\frac{1}{n}\int_M \langle d\theta_\psi,\omega\rangle\,dvol_g}
=
\rev{-\frac{1}{n}\int_M \langle \theta_\psi,d^*\omega\rangle\,dvol_g}
=0.
$$

Set
$$
\mathcal{F}_J=T_J(\mathcal{E}_J).
$$
Since $T_J$ is injective, every $f\in\mathcal{F}_J$ determines a unique $\psi\in\mathcal{E}_J$. We can therefore define
\begin{equation}\label{eq:Dcan-def}
\mathcal{D}^+_J
=
T_J^{-1}
\colon
\mathcal{F}_J
\longrightarrow
\mathcal{E}_J.
\end{equation}
This is a linear bijection. If $f=T_J(\psi)$, then
$$
\mathcal{D}^+_J(f)
=
\rev{dd^*(-f\omega+\sigma_\psi).}
$$

\end{appendices}

  \vskip 6pt

  \noindent{\bf Acknowledgements.}\, \rev{The third author thanks his advisor Z. L\"{u} for his support. The authors also thank Haisheng Liu and Jiuru Zhou for helpful discussions and acknowledge the valuable assistance of AI models in exploring proof strategies and calculations.}

  \vskip 24pt

  \noindent Qiang Tan\\
 School of Mathematical Sciences,
  Jiangsu University, Zhenjiang, Jiangsu 212013, China\\
 e-mail: tanqiang@ujs.edu.cn\\

  \vskip 6pt

\begin{samepage}
  \noindent Hongyu Wang\\
  School of Mathematical Sciences, Yangzhou University, Yangzhou, Jiangsu 225002, China\\
  e-mail: hywang@yzu.edu.cn\\
\par
\end{samepage}

   \vskip 6pt

  \noindent Ken Wang\\
   School of Mathematical Sciences, Fudan University, Shanghai, 100433, China\\
   e-mail: kw.dred@gmail.com

   \vskip 6pt

  \noindent Zuyi Zhang\\
   Beijing International Center for Mathematical Research, Peking University, Beijing, 100871, China\\
   e-mail: zhangzuyi1993@pku.edu.cn

 \end{document}